\documentclass[twoside,11pt,leqno]{article}

\usepackage{amsthm, amssymb, latexsym, amsmath, amscd, array, graphicx, enumerate, lmodern, slantsc, hyperref,amsfonts}
\usepackage[english]{babel}
\usepackage[all]{xy}
\CompileMatrices
\usepackage{mathtools}

\CompileMatrices

\usepackage{color}
\def\red{} 

\def\RP{\protect\operatorname{\mathbb{R}P}}
\def\pps{\protect\operatorname{P_{\mathbf{m}_s}}}
\def\TC{\protect\operatorname{TC}}

\def\zcl{\protect\operatorname{zcl}}

\def\max{\protect\operatorname{max}}
\def\deg{\protect\operatorname{deg}}

\def\Imm{\protect\operatorname{Imm}}
\def\axial{\protect\operatorname{axial}}
\def\secat{\protect\operatorname{secat}}

\newtheorem{proposition}{Proposition}[section]

\newtheorem{theo}[proposition]{Theorem}
\newtheorem{remark}[proposition]{Remark}

\newtheorem{lemma}[proposition]{Lemma}

\newtheorem{example}[proposition]{Example}
\newtheorem{corollary}[proposition]{Corollary}

\newtheorem{conjecture}[proposition]{Conjecture}

\title{Projective product coverings and sequential motion planning algorithms in real projective spaces}
\author{Jes\'us Gonz\'alez\thanks{Partially supported by Conacyt Research Grant 221221.}, \hspace{.2mm} Darwin Guti\'errez\thanks{This paper is based on part of the Ph.D.~thesis work of the second author at the Mathematics Department of the Escuela Superior de F\'isica y Matem\'aticas del Instituto Polit\'ecnico Nacional.}, \hspace{1mm}and \hspace{.6mm}Adriana Lara\thanks{Supported by Grant SIP20152082.}}

\date{\empty}

\begin{document}

\maketitle

\begin{abstract}
For positive integers $m$ and $s$, let $\mathbf{m}_s$ stand for the $s$-th tuple $(m,\ldots,m)$. We show that, for large enough $s$, the higher topological complexity $\TC_s$ of an even dimensional real projective space $\RP^m$ is characterized as the smallest positive integer $k=k(m,s)$ for which there is a $(\mathbb{Z}_2)^{s-1}$-equivariant map from Davis' projective product space $\pps$ to the $(k+1)$-th join-power $((\mathbb{Z}_2)^{s-1})^{\ast(k+1)}$. This is a (partial) generalization of Farber-Tabachnikov-Yuzvinsky's work relating $\TC_2$ to the immersion dimension of real projective spaces. In addition, we compute the exact value of $\TC_s(\RP^m)$ for $m$ even and $s$ large enough.
\end{abstract}




\section{Introduction}\label{secintro}
Michael Farber's notion of topological complexity ($\TC$) was introduced in~\cite{Far,MR2074919} as a way to study the motion planning problem in robotics from a topological perspective. Due to its homotopy invariance, the concept quickly captured the attention of algebraic topologists who began to study the homotopy $\TC$-phenomenology. In particular, Farber's $\TC$ was soon identified as a special instance of a slightly more general concept: Rudyak's higher topological complexity $\TC_s$, which recovers Farber's $\TC$ if $s=2$, can be thought of as a measure of the robustness to noise of motion planning algorithms in automated multitasking processes~(\cite{bgrt,Ru10}). 

\smallskip
Soon after their introduction, the $\TC$-ideas found a highly surprising connection with one of the central problems in last century's main homotopy developments. Namely, it is shown in~\cite{MR1988783} that, for the $m$-dimensional real projective space $\RP^m$, $\TC_2(\RP^m)$ agrees with $\Imm(\RP^m)$, the Euclidean immersion dimension of $\RP^m$, provided $m\neq 1,3,7$. Using the main result in~\cite{MR0336757}, this means that, without restriction on $m$, $\TC_2(\RP^m)$ can be described, in purely homotopic terms, as the minimal positive integer $a(m)$, also denoted by $\axial(\RP^m)$, for which the restriction to $\RP^m\times\RP^m$ of the Hopf multiplication $$\mu\colon\RP^\infty\times\RP^\infty\to\RP^\infty$$ can be compressed to a map $\RP^m\times\RP^m\to\RP^{a(m)}$ ---a so called (optimal) axial map. With this in mind, it is natural to ask for the (geometric and homotopic) properties of $\RP^m$ encoded by the higher analogues $\TC_s(\RP^m)$. Such a task is addressed in this paper and, in doing so, we are naturally lead to Davis' projective product space $\pps$, introduced in~\cite{MR2651360}, and defined as the orbit space of $(S^m)^{\times s}$ by the diagonal (antipodal) $\mathbb{Z}_2$-action ---in Davis' notation, $\mathbf{m}_s$ stands for the $s$-tuple $(m,\ldots,m)$.

\smallskip
In slightly more detail, for $s\geq2$, a natural generalization of the construction in~\cite[(4.2)]{MR1988783} leads to 
\begin{equation}\label{pis}
\TC_s(\RP^m)\geq\secat(\pi_s),
\end{equation}
where $\pi_s\colon\pps\to(\RP^m)^{\times s}$ is the ``pivoted axial'' $(\mathbb{Z}_2)^{\times(s-1)}$-principal bundle whose projection map is induced by the $s$-fold cartesian power of the Hopf double cover $S^m\to\RP^m$ (further details of this construction are reviewed in the next section). The central result in~\cite{MR1988783} asserts that~(\ref{pis}) is an equality for $s=2$. The proof of such a fact is achieved by
\begin{enumerate}[(I)]
\item\label{i1na} connecting $\secat(\pi_2)$ to the existence of (optimal) axial maps $\RP^m\times\RP^m\to\RP^{\secat(\pi_2)}$, and then
\item\label{i2na} showing how (optimal) motion planners for $\RP^m$ are encoded by such axial maps.
\end{enumerate}
It is not difficult to prove the right generalization of~(\ref{i1na}) for $s\geq3$ (see Proposition~\ref{elclasif} below). On the other hand, when $m$ is even, the validness of a suitable statement generalizing~(\ref{i2na}) is hinted both by Proposition~\ref{emp} below and by the cohomological calculations in Section~\ref{seccohpe}. \red{In particular, for $m$ even and $s$ large enough, we prove that equality holds in~(\ref{pis}), and compute the resulting explicit value of $\TC_s(\RP^m)$ ---see Corollary~\ref{partilar} below.} 

\medskip
\red{On the basis of our results, we conjecture that equality always holds in~(\ref{pis}). This would yield a full generalization of Farber-Tabachnikov-Yuzvinsky's result to the higher TC realm. Proving equality in~(\ref{pis}) seems to be inherently more complex when $s\geq3$. See Remarks~\ref{recoverFTY}--\ref{ylaotra} for a discussion of why proving equality in~(\ref{pis}) is elementary for $s=2$, while the corresponding task for $s\geq3$ becomes interestingly more intricate.}

\section{The projective product covering}
\label{sectionpreliminaries}
For an integer $s\geq2$, the \emph{$s$-th higher topological complexity} of a path connected space $X$, $\TC_{s}(X)$, is defined in~\cite{Ru10} as the reduced Schwarz genus of the fibration $$e_s=e^X_{s}:X^{[0,1]}\to X^{s},\qquad e_s(\gamma)=\left(\gamma\left(\frac{0}{s-1}\right),\gamma\left(\frac{1}{s-1}\right),\ldots,\gamma\left(\frac{s-1}{s-1}\right)\right).$$ Thus $\TC_{s}(X)+1$ is the smallest cardinality of open covers $\{U_i\}_i$ of $X^s$ so that $e_s$ admits a continuous section $\sigma_i$ on each $U_i$. The open sets $U_i$ of such an open cover are called {\it $s$-local domains}, the corresponding sections $\sigma_i$ are called {\it $s$-local rules}, and the resulting family of pairs $\{(U_i,\sigma_i)\}$ is called an {\it $s$-motion planning algorithm} for $X$. We say that such a family is an {\it optimal $s$-motion planning algorithm} if it has $\TC_{s}(X)+1$ $s$-local domains. These ideas are a generalization of the concept of topological complexity introduced by Farber in \cite{Far} as a model to study the continuity instabilities in the motion planning of an autonomous system (robot) whose space of configurations is $X$. The term ``higher'' comes by considering the base space $X^s$ of $e_s$ as a series of prescribed stages in the robot motion planning, while Farber's original case $s=2$ deals only with the space $X\times X$ of initial-final stages.

\begin{remark}\label{otromodelo}{\em
As shown in~\cite[pages~2106--2107]{bgrt}, $\TC_s(X)$ can equivalently be defined as the genus of the evaluation map $X^{\Gamma_s}\to X^s$, $\gamma\mapsto(\gamma(v_1),\ldots,\gamma(v_s))$, where $\Gamma_s$ is (the underlying topological space of) a given connected graph, and $v_1,\ldots,v_s$ are $s$ distinct vertices of $\Gamma_s$. In the final section of this paper it will be convenient to take $\Gamma_s$ to be the graph with exactly $s$ vertices $v_1,v_2,\ldots v_s$, and $s-1$ edges $(v_1,v_s),(v_2,v_s),\ldots(v_{s-1},v_s)$ depicted as follows:
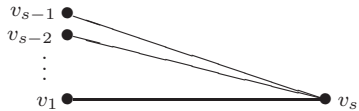
\begin{center}
\begin{picture}(0,50)(-20,-10)
\put(0,0){\line(-1,0){95}}
\put(0,0){\line(-3,1){95}}
\put(0,0){\line(-4,1){95}}
\put(-2.5,-2.5){$\bullet$}
\put(5,-2.5){\scriptsize$v_s$}
\put(-100,-2.5){$\bullet$}
\put(-109,-2.5){\scriptsize$v_1$}
\put(-100,21.7){$\bullet$}
\put(-120,22.5){\scriptsize$v_{s-2}$}
\put(-100,30){$\bullet$}
\put(-107,7.3){\scriptsize $\vdots$}
\put(-119,31.5){\scriptsize$v_{s-1}$}
\end{picture}
\end{center}
}\end{remark}

Most of the existing methods to estimate the higher topological complexity of a space are cohomological in nature. One of the most successful such methods is a special case of Proposition~\ref{coholoweboun} below, which is easily proved on the lines of~\cite[Theorem~4 in page 73]{Schwarz66}.

\begin{proposition}\label{coholoweboun}
Let $h^*$ be a generalized cohomology theory with products. The sectional category of a fibration $\pi\colon E\to B$ is no less than the cup length of elements in the kernel of $\pi^*\colon h^*(B)\to h^*(E)$.
\end{proposition}

Here ``cup-length'' refers to the maximal number of elements in the indicated ideal having a non-vanishing product.

\medskip
Later in the paper we will apply Proposition~\ref{coholoweboun} to the $(\mathbb{Z}_2)^{s-1}$-covering space $\pi_s$ in~(\ref{pis}). The covering space is explicitly defined and studied in this section. Let the group $(\mathbb{Z}_2)^{s-1}$, with obvious generators $\sigma_i$ ($1\leq i\leq s-1$), act on $(S^m)^{\times s}$ so that
\begin{equation}\label{laction}
\sigma_i\cdot(x_1,\ldots,x_s)=(x_1,\ldots,x_{i-1},-x_i,x_{i+1},\ldots,x_{s}).
\end{equation}
Let $\pps$ be the quotient of $(S^m)^{\times s}$ by the involution $\delta\cdot(x_1,\ldots,x_s)=(-x_1,\ldots,-x_s)$. It is elementary to check that the induced $(\mathbb{Z}_2)^{s-1}$-action on $\pps$ is principal and has orbit space $(\RP^m)^{\times s}$. This defines the $(\mathbb{Z}_2)^{s-1}$-principal bundle $\pi_s$.

\smallskip
For a path $\gamma$ in $\RP^m$, pick a lifting $\widetilde{\gamma}$ through the projection $S^m\to\RP^m$, and note that the class of $$\left(\widetilde{\gamma}\left(\frac{0}{s-1}\right),\widetilde{\gamma}\left(\frac{1}{s-1}\right),\ldots,\widetilde{\gamma}\left(\frac{s-1}{s-1}\right)\right)$$ in $\pps$ does not depend on the chosen lifting $\widetilde{\gamma}$. We get a map $(\RP^m)^{[0,1]}\to \pps$ fitting in the commutative diagram
\begin{equation}\label{cdespis}\xymatrix{
(\RP^m)^{[0,1]} \ar[rr] \ar[rd]_{e_s} && \pps \ar[dl]^{\pi_s} \\
& (\RP^m)^{\times s}, &
}\end{equation}
which readily yields~(\ref{pis}). 

\smallskip
The homotopy nature of $\pi_s$ is described through its classifying map as:

\begin{proposition}\label{elclasif}
For $1\leq i\leq s$ let $p_i\colon(\RP^m)^{\times s}\to\RP^m$ be the $i$-th projection, $\xi_m\to\RP^m$ be the Hopf bundle over $\RP^m$, and $\mu_s\colon(\RP^m)^{\times s}\to(\RP^\infty)^{\times (s-1)}$ classify $\pi_s$. Then, for $1\leq i\leq s-1$, the $i$-th component $\mu_{i,s}$ of $\mu_s$ classifies $p_i^*(\xi_m)\otimes p_s^*(\xi_m)$.
\end{proposition}

The conclusion of Proposition~\ref{elclasif} can of course be stated by saying that $\mu_{i,s}$ is homotopic to the composition of the projection $p_{i,s}\colon(\RP^m)^{\times s}\to\RP^m\times\RP^m$ onto the $(i,s)$ coordinates, the inclusion $\RP^m\times\RP^m\hookrightarrow\RP^\infty\times\RP^\infty$,  and the Hopf multiplication $\mu\colon\RP^\infty\times\RP^\infty\to\RP^\infty$.

\begin{proof}[Proof of Proposition~\ref{elclasif}]
Recall $\delta$ stands for the involution $(x_1,\ldots,x_s)\mapsto(-x_1,\ldots,-x_s)$ in $(S^m)^{\times s}$ so, by definition, the corresponding orbit space is $\pps$. The total space $Z_i$ of the $\mathbb{Z}_2$-principal bundle classified by the $i$-th component $\mu_{i,s}$ is the quotient of $(S^m)^{\times s}$ by the actions of $\delta$ and of those $\sigma_\ell$ ($1\leq\ell\leq s-1$) with $\ell\neq i$, and where the $\mathbb{Z}_2$-principal action on $Z_i$ is induced by change of signs on the $i$-th coordinate. Let $\lambda_{i,j}\to\RP^m$ stand for the restriction to the $j$-th axis of the latter double covering (axes are taken with respect to the base point in $\RP^m$ given by the class of $1:=(1,0,\ldots,0)\in S^m$). 

\smallskip\noindent
{\bf Case $j=s\hspace{.3mm}$}: Note that a class in $\lambda_{i,s}$ has a unique representative of the form $(1,\ldots,1,x_s)$ and, in these terms, the $\mathbb{Z}_2$-principal action on $\lambda_{i,s}$ is given by
\begin{align*}
\left[\rule{0mm}{4mm}(1,\ldots,1,x_s)\right]\mapsto\left[(1,\ldots,1,\stackrel{\mbox{\tiny$i$}}{-1},1,\ldots,1,x_s)\right]&=\left[(-1,\ldots,-1,\raisebox{.2mm}{$\stackrel{\mbox{\tiny$i$}}{1}$},-1,\ldots,-1,-x_s)\right]\\ & =\left[(1,\ldots,1,\raisebox{.2mm}{$\stackrel{\mbox{\tiny$i$}}{1}$},1,\ldots,1,-x_s)\right],
\end{align*}
where the notation $\stackrel{\mbox{\tiny$b$}}{a}$ indicates that the number $a$ appears in the $b\,$-th coordinate of the $s$-tuple. Consequently, $\lambda_{i,s}\to\RP^m$ is homeomorphic to the Hopf projection $S^m\to\RP^m$.

\smallskip\noindent
{\bf Case $j=i\hspace{.3mm}$}: As above, a class in $\lambda_{i,i}$ has a unique representative of the form $(1,\ldots,1,\raisebox{.2mm}{$\stackrel{\mbox{\tiny$i$}}{x_i}$},1,\ldots,1)$ and, now, the $\mathbb{Z}_2$-principal action on $\lambda_{i,i}$ is antipodal on $x_i$ on the nose. Thus $\lambda_{i,i}\to\RP^m$ is also homeomorphic to the Hopf projection $S^m\to\RP^m$.

\smallskip\noindent
{\bf Case $j\not\in\{i,s\}$}: Classes in $\lambda_{i,j}$ are represented by elements $(\pm1,\ldots,\pm1,\raisebox{.2mm}{$\stackrel{\mbox{\tiny$j$}}{x_j}$},\pm1,\ldots,\pm1,\raisebox{.2mm}{$\stackrel{\mbox{\tiny$i$}}{\pm1}$},\pm1,\ldots,\pm1)$ where, to fix ideas, we have assumed $j<i<s$ ---the case $i<j<s$ works just as well. Dividing out first by the action of $\delta$ and of the $\sigma_\ell$ with $\ell\not\in\{i,j\}$ (and then by the action of $\sigma_j$), we see that $\lambda_{i,j}$ is given as the quotient of $S^m\times\mathbb{Z}_2$ by the antipodal action on the first coordinate and with $\mathbb{Z}_2$-principal action coming from the antipodal action on the second coordinate. In other words, $\lambda_{i,j}\to\RP^m$ is the trivial $\mathbb{Z}_2$-bundle.

\medskip
The conclusion follows.
\end{proof}

\section{Motion planning algorithms via equivariant maps}\label{sectionjoins}
Recall that the $(k+1)$-iterated self join-power of a topological space $X$, $J_k(X)$, is defined inductively by $J_k(X):=J_{k-1}(X)\ast X$ ($k\geq1$) where $J_0(X)=X$. Then, for a topological group $G$, $B_kG:=J_k(G)/G$ is the $k$-th stage in Milnor's construction of the classifying space $BG:=\left.J_\infty(G)\right/G,$ where $G$ acts diagonally on the vertices of $J_\infty(G):=\bigcup_{k\geq0}J_k(G)$ ---so barycentric coordinates are preserved.

\medskip
In what follows $G_s$ stands for the (discrete) group $(\mathbb{Z}_2)^{\times(s-1)}$. By~\cite[Theorem~9  in page 86]{Schwarz66}, the classifying homotopy class $\mu_s$ in Proposition~\ref{elclasif} has a representative factoring in the form
\begin{equation}\label{fitf}
(\RP^m)^{\times s}\stackrel{\beta_s}{\longrightarrow} B_{\secat(\pi_s)}(G_s)\hookrightarrow B(G_s)\simeq(\RP^\infty)^{\times(s-1)},
\end{equation}
where $\beta_s$ is covered by a $G_s$-equivariant map $\alpha_s\colon\pps\to J_{\secat(\pi_s)}(G_s)$. Then, in terms of the $G_s$-action defined in~(\ref{laction}), the composition of the canonical projection $(S^m)^{\times s}\to\pps$ with $\alpha_s$ yields a $G_s$-equivariant map $\phi_s\colon(S^m)^{\times s}\to J_{\secat(\pi_s)}(G_s)$ satisfying the condition
\begin{equation}\label{condadd}
\phi_s(x_1,\ldots,x_{s-1},-x_s)=\sigma_1\cdots \sigma_{s-1}\cdot\phi_s(x_1,\ldots,x_{s-1},x_s), \hspace{2mm}\mbox{for all $(x_1,\cdots,x_s)\in(S^m)^{\times s}$.}
\end{equation}

\begin{conjecture}\label{lacopr}
An $s$-motion planning algorithm for $\RP^m$ with $\secat(\pi_s)+1$ $s$-local rules can be constructed out of a map $\phi_s$ as above. Consequently $\secat(\pi_s)\geq\TC_s(\RP^m)$, and~(\ref{pis}) becomes an equality for any $s\geq2$.
\end{conjecture}

The conjecture is motivated in part by (the proof of)~\cite[Proposition~6.3]{MR1988783}, which asserts that the case $s=2$ of Conjecture~\ref{lacopr} holds true ---see Propositions~\ref{emp} and Remark~\ref{recoverFTY} below. Corollary~\ref{partilar} in the next section is meant to gather further evidence for the plausibility of Conjecture~\ref{lacopr}. A few additional instances where Conjecture~\ref{lacopr} holds true are included in the final section of this paper. 

\begin{remark}\label{reversible}{\em
One of our main interests in Conjecture~\ref{lacopr} is the possibility of obtaining upper bounds for $\TC_s(\RP^m)$ from the construction of $G_s$-equivariant maps $\phi_s\colon(S^m)^{\times s}\to J_k(G_s)$ satisfying~(\ref{condadd}). Indeed, such a map covers a map $\beta_s$ as in~(\ref{fitf}), so that~\cite[Theorem~9  in page 86]{Schwarz66} implies $k\geq\secat(\pi_s)$, and so $k\geq\TC_s(\RP^m)$ if Conjecture~\ref{lacopr} were to hold.
}\end{remark}

Given spaces $X$ and $Y$, consider the open subspace $U\subset X\ast Y$ consisting of the barycentric expressions $t_0 x +t_1 y$ with ($x\in X$, $y\in Y$, $0\leq t_i$, $t_0+t_1=1$, and) $t_1>0$. Observe that, if $Y$ is discrete, $U$ is a topological disjoint union of open cones with base $X$ (the cones are open in the sense that they are missing their base). In such terms, the following auxiliary result becomes self-evident.

\begin{lemma}\label{compoents}
For $k\geq0$, $s\geq2$, and $0\leq j\leq k$, consider the open set $U_j\subset J_k(G_s)$ consisting of the barycentric expressions $\,\sum_{\ell=0}^kt_\ell g_\ell$ with $t_j>0$ (here, as usual, $g_\ell\in G_s$, $t_\ell\geq0$, and $\sum t_\ell=1$). Then $U_j$ is closed under the action of $G_s$, and has $2^{s-1}\!$ connected components, each of which is open in $U_j$ and contractible (in itself). Further, the induced $G_s$-action on the set of connected components of $U_j$ has a single orbit.
\end{lemma}

\begin{proposition}\label{emp}
Let $D_s=\{(x_1,\ldots,x_s)\in(S^m)^{\times s}\;\colon x_i=x_s\mbox{ for some $i\in\{1,\ldots,s-1$}\}\}$. The conclusions in Conjecture~\ref{lacopr} hold true if one starts with a $G_s$-equivariant map $\phi_s\colon(S^m)^{\times s}\to J_{\secat(\pi_s)}(G_s)$ satisfying~(\ref{condadd}) together with one of the following conditions:
\begin{enumerate}[(1)]
\item For \emph{every} $j\in\{0,1,\ldots,\secat(\pi_s)\}$, $\phi_s(D_s)$ intersects at most a single component of $U_j$.

\vspace{-2mm}
\item For \emph{some} $j_0\in\{0,1,\ldots,\secat(\pi_s)\}$, $\phi_s(D_s)$ is fully contained in some component of $U_{j_0}$.
\end{enumerate}
\end{proposition}

\begin{remark}\label{recoverFTY}{\em
The easy fact that, for $s=2$, there exist maps $\phi_2$ as that assumed in Proposition~\ref{emp} was first noted in~\cite[Lemmas~5.3 and~5.7]{MR1988783}. Explicitly, it is standard that the case $m=1,3,7$ can be accounted by using the multiplication in the complex, quaternion, and octonion numbers, respectively. For $m\neq1,3,7$, since the diagonal inclusion $\RP^m\hookrightarrow\RP^m\times\RP^m$ is a cofibration, any axial map $\alpha\colon\RP^m\times\RP^m\to\RP^{\secat(\pi_2)}$, being nullhomotopic on the diagonal\footnote{This uses the fact that $\secat(\pi_2)>m$, which in turn comes from the assumption $m\neq1,3,7$ (compare to Remark~\ref{morsubt}).}, is homotopic to a map $\alpha'\colon\RP^m\times\RP^m\to\RP^{\secat(\pi_2)}$ which is (necessarily axial and) actually constant on the diagonal. Then any map $\phi_2\colon S^m\times S^m\to J_{\secat(\pi_2)}(\mathbb{Z}_2)=S^{\secat(\pi_2)}$ covering $\alpha'$ is a fortiori constant on the diagonal. In particular, such maps $\phi_2$ satisfy \emph{both} conditions (1) and (2) in Proposition~\ref{emp} for, obviously, the singleton $\phi_2(D_2)$ is fully contained in some component of \emph{each} $U_j$ satisfying $\phi_2(D_2)\cap U_j\neq\varnothing$.
}\end{remark}

\begin{proof}[Proof of Proposition~\ref{emp}]
For $0\leq j\leq \secat(\pi_s)$, set $V_j=\phi_s^{-1}(U_j)\subseteq(S^m)^{\times s}$, and $W_j=q(V_j)\subseteq(\RP^m)^{\times s}$ where $q$ stands for the composition $(S^m)^{\times s}\to\pps\to(\RP^m)^{s}$ of canonical projections. Note that the equality 
\begin{equation}\label{kenfklw}
V_j=q^{-1}(W_j)
\end{equation}
holds since $V_j$ is closed under the action of $\delta$ and of the $\sigma_\ell$ with $1\leq\ell\leq s-1$ (as the $G_s$-equivariant map $\phi_s$ satisfies~(\ref{condadd})). Further, the sets $W_0,\ldots,W_{\secat(\pi_s)}$ form an open cover of $(\RP^m)^{\times s}$.

\smallskip
If condition~(1) in the statement of the proposition holds, we complete the proof by constructing local sections $\varsigma_j\colon W_j\to (\RP^m)^{\Gamma_s}$ ($0\leq j\leq\secat(\pi_s)$) for the evaluation map $(\RP^m)^{\Gamma_s}\to(\RP^m)^{\times s}$ described at the end of Remark~\ref{otromodelo}. Details follow. Fix $j\in\{0,1,\ldots,\secat(\pi_s)\}$. If $\phi_s(D_s)$ intersects $U_j$, let $U_{j,0}$ denote the component of $U_j$ containing $\phi_s(D_s)\cap U_j$; otherwise, choose any component $U_{j,0}$ of $U_j$. For $(L_1,\ldots,L_s)\in W_j$, the $2^s$ elements in $q^{-1}\{(L_1,\ldots,L_s)\}$ lie in $V_j$, in view of~(\ref{kenfklw}). Also, in view of~(\ref{condadd}) and the final assertion in Lemma~\ref{compoents}, exactly two elements in $q^{-1}(L_1,\ldots,L_s)$ have $\phi_s$-image in $U_{j,0}$. Indeed, if one of the latter elements is $(x_1,x_2,\ldots,x_s)$, then the other is $(-x_1,-x_2,\ldots-x_s)$. Furthermore, in these conditions,
\begin{equation}\label{cerezap}
\mbox{if $L_i=L_s$ for some $1\leq i<s$, then in fact $x_i=x_s$,}
\end{equation}
by construction (and in view of Lemma~\ref{compoents}). Then $\varsigma_j(L_1,\ldots,L_s)\colon \Gamma_s\to\RP^m$ is defined to be the map whose restriction to the (oriented) edge $(v_i,v_s)$ describes the uniform-speed motion in $\RP^m$ from $L_i$ to $L_s$ obtained by rotating $L_i$ toward $L_s$ along the plane generated by these two lines (no motion if $L_i=L_s$), and in such a way that the corresponding rotation from $x_i$ to $x_s$ is performed through an angle smaller than $180$ degrees. As shown in the picture below, the latter requirement holds independently of whether one uses $(x_1,\ldots,x_s)$ or $(-x_1,\ldots,-x_s)$, so that $\varsigma_j(L_1,\ldots,L_s)$ is well defined.
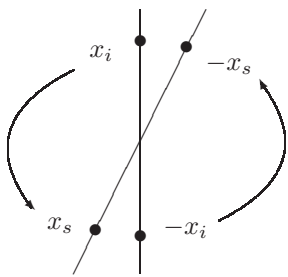
\begin{center}
\begin{picture}(0,98)(0,-48)
\put(0,0){\line(0,1){50}}
\put(0,0){\line(0,-1){50}}
\put(0,0){\line(1,2){25}}
\put(0,0){\line(-1,-2){25}}
\put(-2.5,35.5){$\bullet$}
\put(15,33){$\bullet$}
\put(-2.5,-38.5){$\bullet$}
\put(-19.5,-36){$\bullet$}
\put(-19,32){\small $x_i$}
\put(9,-35){\small $-x_i$}
\put(-35,-33){\small $x_s$}
\put(25,27){\small $-x_s$}
\qbezier(-25,27)(-65,7)(-41,-25)
\put(-39.5,-26){\vector(1,-1){0}}
\qbezier(30,-30)(70,-10)(46,22)
\put(45.4,23.8){\vector(-1,2){0}}
\end{picture}
\end{center}
 The resulting function $\varsigma_j$ is clearly a section on $W_j$ for the evaluation map $(\RP^m)^{\Gamma_s}\to(\RP^m)^s$. Lastly, the continuity of $\varsigma_j$ follows from~(\ref{cerezap}), and from the facts that $U_{j,0}$ is open, that $\phi_s$ is continuous, and that $q$ is a covering projection.
 
 \smallskip
A minor modification of the above construction is needed in order to complete the proof when condition~(2) in the statement of the proposition holds. Indeed, in the notation above, the problematic $q(D_s)$ is contained in $W_{j_0}':=W_{j_0}$, while condition~(2) assures that the construction above yields the needed local section $\zeta'_{j_0}=\zeta_{j_0}\colon W'_{j_0}\to(\RP^m)^{\Gamma_s}$. For all other $j\neq j_0$ we set $W'_j:=W_j-q(D_s)$ (so that the sets $W'_i$ with $0\leq i\leq\secat(\pi_s)$ cover $(\RP^m)^{\times s}$), which (is open and) vacuously avoids the possibility of the failure of~(\ref{cerezap}), thus yielding an obviously continuous local section $\zeta'_j\colon W'_j\to(\RP^m)^{\Gamma_s}$.
\end{proof}

\red{Regarding a potential proof of Conjecture~\ref{lacopr}, the authors believe that, for general $s\geq2$, Proposition~\ref{elclasif} would have to play a key role in proving the existence of a map $\phi_s$ as the one assumed in Proposition~\ref{emp}. However, the problem seems to be much more subtle for $s\geq3$ than the rather straightforward instance $s=2$. We close this section by pinpointing some of the intricacies that are inherent to a potential proof of Conjecture~\ref{lacopr} via Proposition~\ref{elclasif} when $s\geq3$, and how this leads to a couple of interesting new challenges in the field (which we hope to address elsewhere).}

\begin{remark}\label{morsubt}{\em
We start by discussing the relevance of the inequality
\begin{equation}\label{suertota}
\secat(\pi_s\colon\mathrm{P}_{\mathbf{m}_s}\to(\RP^m)^{\times s})\geq(s-1)m,\mbox{ with strict inequality if $m+1$ is not a power of $2$}
\end{equation}
(obtained in~(\ref{pisextendida}) and~Remark~\ref{cisd} below from Proposition~\ref{coholoweboun}) in a potential proof of Conjecture~\ref{lacopr}. Recall that the isomorphism class of the $G_s$-principal bundle $\pi_s$ has been described in Proposition~\ref{elclasif} via the homotopy type of its classifying map $\mu_s\colon(\RP^m)^s\to(\RP^\infty)^{\times(s-1)}$. Of course, the homotopy type of any map $\beta_s\colon(\RP^m)^s\to B_{\secat(\pi_s)}(G_s)$ fitting in the factorization~(\ref{fitf}) does not have to be determined by that of $\mu_s$. Nonetheless, as noted in Remark~\ref{recoverFTY}, a key fact in the proof of the $s=2$ case of Conjecture~\ref{lacopr} is that any such $\beta_2$ remains being null homotopic on the diagonal when $m\neq1,3,7$, as $\secat(\pi_2)>m$ for those values of $m$. (As explained in~\cite[Lemma~5.4]{MR1988783}, the latter inequality turns out to be closely related to Adams' solution of the Hopf invariant~1 problem.) Now, for $s\geq3$, the diagonal is replaced by the ``pivoted'' diagonal $q(D_s)$ used at the end of the proof of Proposition~\ref{emp}. Then, in order to understand the homotopy properties of the restricted $\beta_s|_{D_s}$ from the corresponding properties of the restricted $\mu_s|_{D_s}$ (as in the case $s=2$), we would need to know that $\dim(D_s)$ is {\em strictly smaller} than the connectivity of the inclusion $B_{\secat(\pi_s)}(G_s)\hookrightarrow B_\infty(G_s)\simeq(\RP^\infty)^{\times(s-1)}$. Such a condition is assured by~(\ref{suertota}) if $m+1$ is not a power of 2, as the latter map is a $\secat(\pi_s)$-equivalence (its homotopy fiber agrees with that for the (obviously) $\secat(\pi_s)$-equivalence $J_{\secat(\pi_s)}\to\ast$), while $D_s$ is a union of subcomplexes of $(\RP^m)^{\times s}$ each homeomorphic to $(\RP^m)^{\times(s-1)}$, so that $\dim(D_s)=(s-1)m$. Consequently, the first task to deal with in a proof of Conjecture~\ref{lacopr} based on Proposition~\ref{emp} is to decide whether~(\ref{suertota}) can be improved to a strict inequality when $m+1$ is a power of 2. As indicated in Example~\ref{hopfspaces} below,~(\ref{suertota}) is in fact an equality for $m=1,3,7$, in which case~(\ref{pis}) is an equality too. Thus, the real initial task is to decide whether~(\ref{suertota}) actually improves to a strict inequality for $m=2^e-1$ with $e\geq4$ ---just as in the case $s=2$. A particularly interesting feature of such a challenge is to understand how a potential strict inequality in~(\ref{suertota}) would fit within (a possibly generalized form of) the Hopf invariant~1 problem.
}\end{remark}

\begin{remark}\label{ylaotra}{\em
In addition to the considerations in Remark~\ref{morsubt}, it should be noted that, unlike the situation for $s=2$, no map $\beta_s$ as above can be nullhomotopic on $D_s$ when $s\geq3$ for, in fact, $\mu_s$ evidently fails to be nullhomotopic on $D_s$. Consequently, unlike the situation for $s=2$ discussed in Remark~\ref{recoverFTY}, the issue of being able to ``fix'' a $G_s$-equivariant map $\phi_s$ as in~(\ref{condadd}) so to satisfy at least one of the two conditions in Proposition~\ref{emp} requires handling non-trivial homotopy information.
}\end{remark}

\section{Cohomology estimates}\label{seccohpe}
This section is devoted to estimating the sharpness of~(\ref{pis}) by means of cohomological methods. In particular, we show equality for all even $m$ when $s$ is large enough. Explicitely, an application of Proposition~\ref{coholoweboun} to $e_s$, which is a fibrational replacement for the diagonal $\Delta_s\colon X\hookrightarrow X^{\times s}$, yields the lower bound $$\TC_s(X)\geq\zcl_s^{h^*}(X),$$ where $\zcl_s^{h^*}(X)$ is the $h^*$-cup-length of $s$-th zero-divisors in $X$, i.e.~of elements in the kernel of the induced map $\Delta_s^*\colon h^*(X^{\times s})\to h^*(X)$ (see~\cite[Definition~3.8]{bgrt}). In this section we show that, when $X:=\RP^m$ and $h^*:=H^*$ is singular cohomology with mod 2 coefficients, $\zcl_s(\RP^m):=\zcl_s^{H^*}(\RP^m)$ is in fact a lower bound for the right hand-side in~(\ref{pis}), which, for $m$ odd and $s$ large enough, agrees with the well known upper bound $sm\geq\TC_s(\RP^m)$ coming from~\cite[Theorem~3.9]{bgrt}.

\medskip Recall that $H^*((\RP^m)^{\times s})=H^*(\RP^m)^{\otimes s}$ is the $\mathbb{Z}_2$-algebra generated by the classes $x_i=p_i^*(x)$ subject to the relations $x_i^{m+1}=0$, $1\leq i\leq s$, where $x\in H^1(\RP^m)$ is the first Stiefel-Whitney class of $\xi_m$, and $p_i$ is defined in Proposition~\ref{elclasif}. We do not stress the dependence of $x_i$ on $s$ because, if $s'>s$ and $\pi_{s,s'}\colon(\RP^m)^{\times s'}\to(\RP^m)^{\times s}$ is the projection onto the first $s$ coordinates, then we think of the map induced in cohomology by $\pi_{s,s'}$ as a honest inclusion. The standard (graded) basis of $H^*((\RP^m)^{\times s})$ consists of all the monomials $x_1^{a_1}x_2^{a_2}\cdots x_s^{a_s}$ with $0\leq a_i\leq m\,$. Note that each $x_i+x_s$ ($1\leq i\leq s-1$) is an $s$-th zero-divisor, so it pulls back trivialy under the evaluation map $e_s^*$. In fact:

\begin{proposition}\label{eulerclass}
For $1\leq i\leq s-1$, $x_i+x_s$ pulls back trivially under the map $\pi_s$ on the right of~(\ref{cdespis}).
\end{proposition}
\begin{proof}
The projection $(S^m)^{\times s}\to S^m\times S^m$ onto the $(i,s)$ axes induces a map from $\pi_s$ to $\pi_2$ liying over $p_{i,s}$. The conclusion then follows since $x\otimes1+1\otimes x\in H^1(\RP^m\times RP^m)$, the mod 2 Euler class of the exterior product $\xi_m\otimes\xi_m$, vanishes under $\pi_2$, which is the sphere bundle of $\xi_m\otimes\xi_m$.
\end{proof}

\begin{lemma}\label{generators}
The ideal of $s$-th zero-divisors in $H^*(\RP^m)^{\otimes s}$ is generated by the elements $x_i+x_s$ in Proposition~\ref{eulerclass}.
\end{lemma}
\begin{proof}
Let $\sum_{(a_1,\ldots,a_s)} x_1^{a_1}\cdots x_s^{a_s}$
be the expression  of an homogeneous $s$-th zero-divisor $z$ in terms of the standard basis. Note that the number of summands must be even if $\deg(z)\leq m$. Thus, it suffices to prove that the following elements lie in the ideal $I_s$ generated by the binomials $x_i+x_s$:
\begin{enumerate}[(i)]
\item\label{abajo} The sum of any two basis elements in degree at most $m$.
\item\label{arriba} A basis element in degree greater than $m$.
\end{enumerate}
Elements in~(\ref{abajo}) are easily dealt with by induction on the degree and on the number of common factors. For instance
$$
x_1x_2+x_3x_4=(x_1x_2+x_2x_3)+(x_2x_3+x_3x_4)=x_2(x_1+x_3)+x_3(x_2+x_4).
$$
Elements in~(\ref{arriba}) are dealt with also by an inductive argument based on the fact that, for $i<j$,
$$
x_i^{a_i}x_{i+1}^{a_{i+1}}\cdots x_j^{a_j}=(x_i+x_j)\cdot x_i^{a_i-1}x_{i+1}^{a_{i+1}}\cdots x_j^{a_j}+x_i^{a_i-1}x_{i+1}^{a_{i+1}}\cdots x_{j-1}^{a_{j-1}}x_j^{a_j+1}\equiv x_i^{a_i-1}x_{i+1}^{a_{i+1}}\cdots x_{j-1}^{a_{j-1}}x_j^{a_j+1},
$$
where the congruence holds module $I_s$.
\end{proof}

Thus (\ref{pis}) extends to
\begin{equation}\label{pisextendida}
sm\geq\TC_s(\RP^m)\geq\secat(\pi_s)\geq\zcl_s(\RP^m).
\end{equation}
Set $G(m,s)=sm-\zcl_s(\RP^m)$, so equality holds in~(\ref{pis}) whenever $G(m,s)=0$.
\begin{lemma}\label{Gsequencepaperform}
$G(m,2)\geq G(m,3)\geq G(m,4) \geq \cdots \geq0$.
\end{lemma}
\begin{proof}
Thinking in terms of the expression of elements as sums of the standard basis of $H^*(\RP^m)^{\otimes s}$, we see that if $z\in H^*((\RP^m)^{\times s})$ is a non-zero product of $s$-th zero-divisors, then $$z\cdot (x_1+x_{s+1})^m=z\cdot(x_{s+1}^m+\cdots)\neq0.$$ So, $\zcl_{s+1}(\RP^m)\geq\zcl_s(\RP^m)+m$ and the result follows.
\end{proof}

\begin{remark}\label{cisd}{\em
It is elementary to check that $\zcl_2(\RP^m)=2^{z(m)}-1$, where $z(m)$ is the integral part of $\log_2(2m)$ (c.f.~\cite[Theorem~4.5]{MR1988783}). The last line in the previous proof then implies $\zcl_s(\RP^m)\geq(s-1)m$ with strict inequality if $m+1$ is not a power of 2. The proof of Theorem~\ref{determinationofgapsversionpaper} below is based on a streamlined version of the previous inequality.
}\end{remark}

The monotonic sequence of non-negative integers in Lemma~\ref{Gsequencepaperform} stabilizes, and we denote by $G(m)$ the corresponding stable value. 

\begin{theo}\label{determinationofgapsversionpaper}
Assume $m\equiv2^e-1\!\!\!\mod2^{e+1}$ with $e\geq0$. In other words, $e$ is the length of the block of consecutive ones ending the binary expansion of $m$. For instance, $e=0$ if and only if $m$ is even. Then $G(m)\leq2^e-1$ with equality if $m$ is even, or if $m=2^e-1$. In fact, $G(m,s)\leq2^e-1$ for $s\geq\max\{(m+1)/2^e,2\}$. Specifically, if $m>2^e-1$ and $\sigma$ stands for $(m+1)/2^e$ (so $\sigma$ is an integer greater than $2$), then the product of $\sigma$-th zero-divisors
$$
(x_1+x_\sigma)^{m+2^e}\cdots(x_{\sigma-1}+x_\sigma)^{m+2^e}\in H^*((\RP^m)^\sigma)
$$
is non-zero.
\end{theo}

\begin{conjecture}\label{conejura}
In Theorem~\ref{determinationofgapsversionpaper}, the equality\hspace{.3mm}\footnote{Conjecture~\ref{conejura} has recently been proved in~\cite[Theorem~3.3]{CGG}.} $G(m)=2^e-1$ holds without restriction on $e$.
\end{conjecture}

\begin{example}\label{hopfspaces}{\em
For $e\geq1$ and $s\geq2$,
\begin{equation}\label{potenciasde2}
0\neq x_1^{2^e-1} x_2^{2^e-1}\cdots x_{s-1}^{2^e-1}+\cdots =(x_1+x_s)^{2^e-1}(x_2+x_s)^{2^e-1}\cdots(x_{s-1}+x_s)^{2^e-1}\in H^*((\RP^{2^e-1})^{\times s}),
\end{equation}
which yields 
$G(2^e-1,s)\leq2^e-1$. The latter inequality is in fact an equality in view of Lemma~\ref{generators} and the fact that the $2^e$-th power of any element in $H^*((\RP^{2^e-1})^{\otimes s})$ vanishes. In the case of the three Hopf spaces $\RP^1$, $\RP^3$, and $\RP^7$, the $\TC_s(\RP^m)$-contents of the assertion $G(2^e-1,s)=2^e-1$ is strengthened by~\cite[Theorem~1]{MR3020869}: $\TC_s(\RP^m)=m(s-1)$ for all $s$ and $m\in\{1,3,7\}$. Thus, for all $s\geq2$ and $m\in\{1,3,7\}$, the last three terms in~(\ref{pisextendida}) are all equal to $m(s-1)$.
}\end{example}

\begin{proof}[Proof of Theorem~\ref{determinationofgapsversionpaper}]
The case $m=2^e-1$ is accounted for in Example~\ref{hopfspaces}, so we assume $m>2^e-1$. The hypothesis on $m$ and $e$ implies that the binomial coefficient $\binom{m+2^e}{2^e}$ is odd, so
$$
(x_i+x_\sigma)^{m+2^e}=x_i^m x_\sigma^{2^e} +\text{\em terms involving powers $x_i^j$ with $j<m$}
$$
for $2\leq i\leq\sigma$. Therefore, ignoring basis elements $x_1^{a_1}\cdots x_\sigma^{a_\sigma}$ having $a_i<m$ for some $i\in\{1,\ldots,\sigma-1\}$, the product of $\sigma$-th zero-divisors under consideration becomes $$(x_1^mx_\sigma^{2^e})(x_2^mx_\sigma^{2^e})\cdots(x_{\sigma-1}^mx_\sigma^{2^e})=x_1^mx_2^m\cdots x_{\sigma-1}^mx_\sigma^{(\sigma-1)2^e},$$ which is an element of the standard basis.
\end{proof}

Corollary~\ref{partilar} below, a direct consequence of Theorem~\ref{determinationofgapsversionpaper}, should be compared with the final assertion in Example~\ref{hopfspaces}.

\begin{corollary}\label{partilar}
If $m$ is even and $s>m$, all inequalities in~(\ref{pisextendida}) are in fact equalities.
\end{corollary}

The hypothesis $s>m$ can substantially be relaxed in many cases. For instance,~\cite[Theorem~1.2]{GGGL} implies that the conclusion in Corollary~\ref{partilar} remains true for all $s\geq3$ if $m$ is a 2-power. Other concrete instances follow from Propositions~4.2, 4.7 and 4.9--4.12 in~\cite{CGG}.

\section{Examples with $\TC_s(\RP^m)=\secat(\pi_s\colon\mathrm{P}_{\mathbf{m}_s}\to(\RP^m)^{\times s})$}
In this brief closing section we summarize our knowledge of examples where~(\ref{pis}) is either an equality, or holds within one from being so. On the other hand, we are not aware of any case where~(\ref{pis}) actually fails to be an equality.

\medskip
Since $\TC_s(\RP^1)=s-1$ (\cite[Corollary~3.12]{bgrt}),~(\ref{suertota}) and~(\ref{pisextendida}) force~(\ref{pis}) to be an equality for $m=1$. In slightly more general terms, and as indicated in Example~\ref{hopfspaces}, equality in~(\ref{pis}) holds for  $m\in\{1,3,7\}$. It would be interesting to give an explicit construction of the corresponding (forced) $G_s$-maps $\phi_s\colon (S^m)^{\times s}\to J_{s-1}(G_s)$ satisfying~(\ref{condadd}). For instance, when $s=2$ and $m=1$, so that $J_{s-1}(G_s)=S^1$, the required map $\phi_2$ can be defined by multiplication of complex numbers. 

\medskip
In the previous section we have discussed how Theorem~\ref{determinationofgapsversionpaper} provides instances with equality in~(\ref{pis}) when $m$ is even. We now remark that the same arguments show that, in any case,~(\ref{pis}) fails from being an equality by at most a unit provided $m\equiv1\bmod4$ and $s\geq\frac{m+1}{2}$ (as in the case of $m$ even, the restriction imposed by the last inequality can usually be relax substantially).


\bigskip
{\sc Departamento de Matem\'aticas

Centro de Investigaci\'on y de Estudios Avanzados del Instituto Polit\'ecnico Nacional

Av.~IPN 2508, Zacatenco, M\'exico City 07000, M\'exico

{\tt jesus@math.cinvestav.mx}

\bigskip
Departamento de Formaci\'on B\'asica

Escuela Superior de C\'omputo del Instituto Polit\'ecnico Nacional

Juan de Dios B\'atiz esq.~Miguel Oth\'on de Mendiz\'abal, M\'exico City 07738, Mexico.

{\tt dargut@hotmail.com}

\bigskip
Departamento de Matem\'aticas

Escuela Superior de F\'isica y Matem\'aticas del Instituto Polit\'ecnico Nacional

Edificio 9, U.P.~Adolfo L\'opez Mateos, Mexico City 07300, Mexico.
               
{\tt adriana@esfm.ipn.mx}
}

\end{document}